\documentclass[12pt,dvipdfmx]{article}
\usepackage{algorithm}
\usepackage{algpseudocode}
\usepackage{amsfonts}
\usepackage{amsmath}
\usepackage{amssymb} 
\usepackage{amsthm} 
\usepackage[title]{appendix} 
\usepackage{autobreak}
\usepackage{bigdelim} 
\usepackage{bm}
\usepackage{breqn}
\usepackage{cancel} 
\usepackage{cases}
\usepackage{chngcntr}
\usepackage{colortbl} 
\usepackage{comment}
\usepackage{empheq} 
\usepackage{enumerate}
\usepackage{fancyhdr} 
\usepackage{fancyvrb} 
\usepackage{framed} 
\usepackage{graphicx}
\usepackage{hyperref}
\hypersetup{
  colorlinks,
  linkcolor={red!50!black},
  citecolor={blue!50!black},
  urlcolor={blue!80!black}
}
\usepackage{import}
\usepackage{mathrsfs} 
\usepackage{multirow} 
\usepackage{setspace} 
\usepackage{subfig} 
\usepackage{tikz} 
\usepackage{pgfplots}
\usepackage{pxrubrica} 
\usepackage{url}

\usepackage[top=35truemm,bottom=35truemm,left=30truemm,right=30truemm]{geometry} 
\makeatletter

\@addtoreset{equation}{section}
\makeatother

\algtext*{EndWhile}
\algtext*{EndIf}
\algtext*{EndFor}
\newtheorem{theorem}{Theorem}[section]
\newtheorem{proposition}[theorem]{Proposition}%
\newtheorem{corollary}[theorem]{Corollary}
\newtheorem{lemma}[theorem]{Lemma}

\newtheorem{remark}[theorem]{Remark}%
\newcommand{\relmiddle}[1]{\mathrel{}\middle#1\mathrel{}}

\newcommand{\calK}{\mathcal{K}}

\newcommand{\calN}{\mathcal{N}}

\newcommand{\calS}{\mathcal{S}}

\newcommand{\DNN}{\mathcal{DNN}}
\newcommand{\CP}{\mathcal{CP}}
\newcommand{\COP}{\mathcal{COP}}
\newcommand{\SPN}{\mathcal{SPN}}
\newcommand{\DD}{\mathcal{DD}}
\newcommand{\SDD}{\mathcal{SDD}}

\newcommand{\RNum}[1]{\uppercase\expandafter{\romannumeral #1\relax}} 
\newcommand{\Rnum}[1]{\lowercase\expandafter{\romannumeral #1\relax}} 

\title{On the longest chain of faces of \\the completely positive and copositive cones}

\makeatletter
\let\@fnsymbol\@arabic
\makeatother

\author{
\normalsize
    Mitsuhiro Nishijima\thanks{Department of Industrial Engineering and Economics,
        Tokyo Institute of Technology, 2-12-1 Ookayama, Meguro-ku, 1528552, Tokyo, Japan; The Institute of Statistical Mathematics, 10-3 Midori-cho, Tachikawa-shi, 1908562, Tokyo, Japan.
        ({\tt nishijima.m.ae@m.titech.ac.jp}).}
        }

\begin{document}
\maketitle

\begin{abstract}\noindent
We consider a wide class of closed convex cones $\mathcal{K}$ in the space of real $n\times n$ symmetric matrices and establish the existence of a chain of faces of $\mathcal{K}$, the length of which is maximized at $\frac{n(n+1)}{2} + 1$.
Examples of such cones include, but are not limited to, the completely positive and the copositive cones.
Using this chain, we prove that the distance to polyhedrality of any closed convex cone $\mathcal{K}$ that is sandwiched between the completely positive cone and the doubly nonnegative cone of order $n \ge 2$, as well as its dual, is at least $\frac{n(n+1)}{2} - 2$, which is also the worst-case scenario.
\end{abstract}
\vspace{0.5cm}


\section{Introduction}\label{sec:intro}
A completely positive cone and its dual, the copositive cone, have been a subject of considerable attention in the optimization field, as they can be used to represent many NP-hard problems.
The recent survey~\cite{DR2021} and references therein provide an overview of their properties and applications.
However, solving optimization problems with these cones is computationally challenging, and as a result, tractable approximations are often used in practice.
Two such examples are the doubly nonnegative cone and the cone of sums of a positive semidefinite matrix and a nonnegative matrix, also known as the SPN cone~\cite{YM2010}.

In this paper, we focus on the length $\ell_{\mathcal{K}}$ of the longest chain of faces and the distance $\ell_{\rm poly}(\mathcal{K})$ to polyhedrality for several matrix cones related to the aforementioned ones.
These lengths are related to various quantities that arise in linear algebra and optimization.
The first example is the \emph{Carath\'{e}odory number}, which can be viewed as the maximum \emph{cp-rank} of completely positive matrices for the completely positive cone.
Ito and Louren\c{c}o~\cite{IL2017} showed that $\ell_{\mathcal{K}} - 1$ is an upper bound for the Carath\'{e}odory number of a closed convex cone $\mathcal{K}$ containing no lines.
The second example is the \emph{singularity degree}.
The singularity degree of a linear conic feasibility problem is the minimum number of facial reduction steps needed to satisfy Slater's condition.
This quantity is known to be related to error bounds for the feasibility problem~\cite{LLP2023,Lourenco2021,Sturm2000}.
Waki and Muramatsu~\cite[Corollary~3.1]{WM2013} showed that the singularity degree of a linear conic feasibility problem over a closed convex cone $\mathcal{K}$ is bounded by $\ell_{\mathcal{K}}$ (see also \cite[Theorem~28.5.3]{Pataki2013}, \cite[Theorem~1]{LP2018}, and \cite[page~2315]{LMT2018}).
Louren\c{c}o, Muramatsu, and Tsuchiya~\cite[Theorem~10]{LMT2018} improved this bound and revealed that the singularity degree is bounded by $\ell_{\rm poly}(\mathcal{K}) + 1$.

The main results of this paper are disappointing for these quantities because the results imply that the upper bounds for them shown in the previous paragraph agree with trivial ones.
We prove that for any closed convex cone $\mathcal{K}$ sandwiched between the completely positive cone and the doubly nonnegative cone, the lengths $\ell_{\mathcal{K}}$ and $\ell_{\rm poly}(\mathcal{K})$ coincide with their trivial upper bounds (Corollary~\ref{cor:dist_poly_CP}).
Specifically, for such a cone $\mathcal{K}$ in the space of real $n\times n$ symmetric matrices, $\ell_{\mathcal{K}}$ equals $\frac{n(n+1)}{2} + 1$ and if $n\ge 2$, $\ell_{\rm poly}(\mathcal{K})$ equals $\frac{n(n+1)}{2} - 2$, which are the worst cases possible for these lengths.
The same holds for the dual cone (Corollary~\ref{cor:dist_poly_COP}).
To the best of our knowledge, these lengths, except for the length of the longest chain of faces of the doubly nonnegative cone~\cite[Proposition~21]{LMT2018}, have not been computed so far (see~\cite[Table~1]{IL2017}).

The cones $\mathcal{K}$ for which the above results hold are not limited to the completely positive, doubly nonnegative, SPN, and copositive cones: for example, the closure of the \emph{completely positive semidefinite cone}~\cite{BLP2017,LP2015} and the approximation hierarchies provided by Parrilo~\cite{Parrilo2000} and Pe\~{n}a, Vera, and Zuluaga~\cite{PVZ2007}.
Moreover, the result of length $\ell_{\mathcal{K}}$ holds for a wider class of $\mathcal{K}$, specifically for any closed convex cone sandwiched between the nonnegative diagonally dominant cone and the nonnegative cone (Proposition~\ref{prop:longest_chain_CP}), or between the nonnegative cone and the dual of the nonnegative scaled diagonally dominant cone (Proposition~\ref{prop:longest_chain_COP}).

We construct a chain of faces by specifying the pattern of zeros of the elements of the matrices belonging to each face to prove our main results.
The idea is inspired by the work of Louren\c{c}o, Muramatsu, and Tsuchiya~\cite{LMT2018}, but we use a different order in which the elements are restricted to zero.
To compute the distance to polyhedrality, it is necessary to construct a chain of faces such that as few polyhedral cones as possible appear (see Remark~\ref{rem:difference}).

The rest of this paper is organized as follows.
In Section~\ref{sec:preliminaries}, the notations and some lemmas used in this paper are introduced.
The main results are proven in Sections~\ref{sec:CP} and \ref{sec:COP}.
In Section~\ref{sec:conclusion}, we finish with concluding remarks.

\section{Preliminaries}\label{sec:preliminaries}
Matrices are denoted by boldface letters such as $\bm{A}$.
The $(i,j)$th element of a matrix $\bm{A}$ is expressed as $A_{ij}$.
Let $\bm{O}$ be a zero matrix with appropriate size.
$\calS^n$ is used to denote the space of real $n\times n$ symmetric matrices, and we define $T_n \coloneqq \frac{n(n+1)}{2}$.

A set $\mathcal{K} \subseteq \calS^n$ is called a \emph{cone} if $\alpha\bm{A} \in \mathcal{K}$ for all $\alpha \ge 0$ and $\bm{A}\in\mathcal{K}$.
Let $\mathcal{K}$ be a closed convex cone in $\calS^n$.
The set
\begin{equation*}
\mathcal{K}^* \coloneqq \left\{\bm{A}\in \calS^n \relmiddle| \sum_{i,j=1}^nA_{ij}B_{ij} \ge 0 \text{ for all $\bm{B}\in \mathcal{K}$}\right\}
\end{equation*}
is called the dual of $\mathcal{K}$.
Let $\mathcal{I}$ be a subset of $\{(i,j) \mid i,j = 1,\dots,n\}$.
Then we define $\mathcal{K}[\mathcal{I}] \coloneqq \{\bm{A}\in\mathcal{K} \mid A_{ij} = 0 \text{ for all $(i,j)\in\mathcal{I}$}\}$ and $\bm{E}[\mathcal{I}] \in \calS^n$ as the matrix with the $(i,j)$th and $(j,i)$th elements $0$ for each $(i,j)\in \mathcal{I}$ and $1$ for all other elements.

Let $\mathbb{R}_+^n$ be the set of $n$-dimensional nonnegative vectors.
We define
\begin{align*}
\CP^n &\coloneqq \left\{\sum_{i=1}^m\bm{a}_i\bm{a}_i^\top \relmiddle|\bm{a}_i\in\mathbb{R}_+^n \text{ for all $i = 1,\dots,m$}\right\},\\
\COP^n &\coloneqq \{\bm{A}\in\calS^n \mid \bm{x}^\top\bm{A}\bm{x} \ge 0 \text{ for all $\bm{x}\in\mathbb{R}_+^n$}\}
\end{align*}
and call them the \emph{completely positive cone} and the \emph{copositive cone}, respectively.
We use $\calS_+^n$ to denote the set of $n\times n$ symmetric positive semidefinite matrices.
The \emph{nonnegative cone} $\mathcal{N}^n$ is defined by the set of $n\times n$ symmetric matrices with only nonnegative elements.
The \emph{doubly nonnegative cone} $\calS_+^n \cap \mathcal{N}^n$ is denoted by $\DNN^n$ and the \emph{SPN cone} $\calS_+^n + \mathcal{N}^n$ by $\SPN^n$.
Let
\begin{equation*}
\DD^n \coloneqq \left\{\bm{A}\in \calS^n \relmiddle| A_{ii} \ge \sum_{j\neq i}\lvert A_{ij}\rvert \text{ for all $i = 1,\dots,i$}\right\}
\end{equation*}
denote the set of $n\times n$ symmetric diagonally-dominant matrices with nonnegative diagonal elements and let $\SDD^n$ be the set of matrices $\bm{D}\bm{A}\bm{D}$ such that $\bm{A}\in \DD^n$ and $\bm{D}\in \calS^n$ is a diagonal matrix with positive diagonal elements.
The \emph{nonnegative diagonally dominant cone} $\DD^n \cap \mathcal{N}^n$ is denoted by $\DD_+^n$, and the \emph{nonnegative scaled diagonally dominant cone} $\SDD^n \cap \mathcal{N}^n$ is denoted by $\SDD_+^n$.
These two cones were originally introduced by Gouveia, Pong, and Saee~\cite{GPS2020} to provide inner-approximation hierarchies for $\CP^n$.
The set $\DD_+^n$ is known to be the conical hull of the matrices $\bm{E}_{ij} \coloneqq (\bm{e}_i + \bm{e}_j)(\bm{e}_i + \bm{e}_j)^\top$ for $i,j = 1,\dots,n$~\cite[page~390]{GPS2020}, where $\bm{e}_i$ is the vector with the $i$th element 1 and the others 0.

The nonnegative cone is self-dual, i.e., $(\mathcal{N}^n)^* = \mathcal{N}^n$.
In addition, it is known that $\CP^n$ and $\COP^n$ are mutually dual~\cite[Theorem~3.28]{SB2021} and that $\DNN^n$ and $\SPN^n$ are also dual~\cite[Theorem~1.167]{SB2021}.
The dual of $\DD_+^n$ is
\begin{align}
(\DD_+^n)^* &= (\DD^n)^* + (\mathcal{N}^n)^* \nonumber\\
&=\{\bm{A} \in \calS^n \mid A_{ii} + A_{jj} \pm 2A_{ij} \ge 0 \text{ for all $i,j = 1,\dots,n$}\} + \mathcal{N}^n, \label{eq:dual_DD+}
\end{align}
where we use Corollaries 16.4.2 and 9.1.3 of \cite{Rockafellar1970} to derive the first equation and use \cite[Table~1]{PP2018} to derive the second equation.
Similarly, the dual of $\SDD_+^n$ is
\begin{align}
&(\SDD_+^n)^* \nonumber\\
&\quad= (\SDD^n)^* + (\mathcal{N}^n)^* \nonumber\\
&\quad= \left\{\bm{A} \in \calS^n \relmiddle|
\begin{aligned}
&A_{ii}A_{jj} \ge A_{ij}^2 \text{ for all $i,j = 1,\dots,n$ with $i\neq j$},\\
&A_{ii} \ge 0 \text{ for all $i = 1,\dots,n$}
\end{aligned}
\right\} + \mathcal{N}^n, \label{eq:dual_SDD+}
\end{align}
where we use \cite[Proposition~\mbox{1.\Rnum{2}}]{GPS2020} to derive the first equation and use \cite[Table~1]{PP2018} to derive the second equation.

Equation~\eqref{eq:dual_SDD+} leads to the following lemma.
Its proof is straightforward and thus omitted.
\begin{lemma}\label{lem:dual_SDD+}
Let $\bm{A} \in (\SDD_+^n)^*$.
Then the following statements hold:
\begin{enumerate}[(i)]
\item $A_{ii} \ge 0$ for all $i = 1,\dots,n$. \label{enum:dual_SDD+_diag}
\item If $A_{ii} = 0$, then $A_{ij} \ge 0$ for all $j = 1,\dots,n$. \label{enum:dual_SDD+_nondiag}
\end{enumerate}
\end{lemma}

The following inclusions hold:
\begin{equation}
\DD_+^n \subseteq \SDD_+^n \overset{\scriptsize \text{(a)}}\subseteq \CP^n \overset{\scriptsize \text{(b)}}{\subseteq} \DNN^n \subseteq \mathcal{N}^n, \label{eq:inclusion_CP}
\end{equation}
where the inclusion~(a) follows from \cite[Proposition~\mbox{1.\Rnum{3}}]{GPS2020} and the others follow from their definitions.
Taking their dual, we also obtain the following inclusions:
\begin{equation}
\mathcal{N}^n \subseteq \SPN^n \overset{\scriptsize \text{(c)}}\subseteq \COP^n \subseteq (\SDD_+^n)^* \subseteq (\DD_+^n)^*. \label{eq:inclusion_COP}
\end{equation}
The inclusions~(b) and (c) hold with equality if and only if $n \le 4$; see Sections~\mbox{2.9} and \mbox{2.10} of \cite{SB2021}.
We state this fact explicitly because it will be exploited to compute the distance to polyhedrality.

\begin{lemma}\label{lem:equation}
The equations $\CP^n = \DNN^n$ and $\SPN^n = \COP^n$ hold if and only if $n \le 4$.
\end{lemma}

Let $\mathcal{K}$ be a closed convex cone in $\calS^n$.
A nonempty convex subcone $\mathcal{F}$ of $\mathcal{K}$ is called a \emph{face} of $\mathcal{K}$ if for any $\bm{A},\bm{B}\in\mathcal{K}$, if $\bm{A} + \bm{B} \in \mathcal{F}$, then $\bm{A},\bm{B}\in\mathcal{F}$.
A \emph{chain of faces} of $\mathcal{K}$ with length $l$ is a sequence
\begin{equation}
\mathcal{F}_l \subsetneq \cdots \subsetneq \mathcal{F}_1 \label{eq:chain}
\end{equation}
such that every $\mathcal{F}_i$ is a face of $\mathcal{K}$.
We use $\ell_{\mathcal{K}}$ to denote the \emph{length of the longest chain of faces} of $\mathcal{K}$.
In addition, we define the \emph{distance $\ell_{\rm poly}(\mathcal{K})$ to polyhedrality} of $\mathcal{K}$ as the length \emph{minus one} of the longest chain \eqref{eq:chain} of faces of $\mathcal{K}$ such that $\mathcal{F}_l$ is polyhedral and every $\mathcal{F}_i$ with $i < l$ is not polyhedral.

The distance to polyhedrality was originally introduced in \cite{LMT2018} to bound the number of reduction steps in the facial reduction algorithm described therein.
The value $\ell_{\rm poly}(\mathcal{K})$ is $0$ if $\calK$ is a polyhedral cone such as $\calN^n$ and $\DD_+^n$.
It is greater than $0$ if $\calK$ is not polyhedral.
For instance, we have $\ell_{\rm poly}(\calS_+^n) = n-1$~\cite[Example~1]{LMT2018}.
See \cite[Remark~\mbox{39}]{Lourenco2021} and \cite[Example~1]{LMT2018} for further results on the distance to polyhedrality.

The following two lemmas provide the upper bounds for $\ell_{\mathcal{K}}$ and $\ell_{\rm poly}(\mathcal{K})$, respectively.

\begin{lemma}[{see~\cite[Proposition~21]{LMT2018}}]\label{lem:length_upper_bound}
If $\mathcal{K}$ is a closed convex cone in $\calS^n$, then $\ell_{\mathcal{K}} \le T_n + 1$.
\end{lemma}

\begin{lemma}\label{lem:dist_poly_upper_bound}
Let $\mathcal{K}$ be a closed convex cone in $\calS^n$.
\begin{enumerate}[(i)]
\item If $n = 1$, then $\ell_{\rm poly}(\mathcal{K}) = 0$. \label{enum:n=1}
\item If $n \ge 2$, then $\ell_{\rm poly}(\mathcal{K}) \le T_n - 2$. \label{enum:n_ge_2}
\end{enumerate}
\end{lemma}

\begin{proof}
First, note that every closed convex cone with dimension at most two is polyhedral (see, for example, \cite[Exercise~1.65]{SB2021}).
If $n = 1$, then the dimension of $\mathcal{K}$ is at most one, and thus item~\eqref{enum:n=1} holds.
In what follows, we prove item~\eqref{enum:n_ge_2}.
To obtain a contradiction, we assume that $l \coloneqq \ell_{\rm poly}(\mathcal{K}) \ge T_n - 1$.
It follows from $n \ge 2$ that $l \ge T_n - 1 \ge 2$.
Then there exists a chain $\mathcal{F}_{l+1} \subsetneq \cdots \subsetneq \mathcal{F}_1$ of faces of $\mathcal{K}$ such that $\mathcal{F}_{l}$ is not polyhedral.
As $\dim(\mathcal{F}_l) < \cdots < \dim(\mathcal{F}_1)$ and each value must be an integer between $0$ and $T_n$, we have $\dim(\mathcal{F}_l) \le T_n - l + 1 \le 2$.
Then $\mathcal{F}_l$ is polyhedral, which is a contradiction.
\end{proof}

\section{The completely positive side}\label{sec:CP}
For $i = 1,\dots,n$ and $j = i,\dots,n$, let
\begin{equation*}
\mathcal{I}_{i,j} \coloneqq \bigcup_{k=1}^{i-1}\bigcup_{l=k}^n\{(k,l)\} \cup \bigcup_{l=j}^n\{(i,l)\}.
\end{equation*}
Here, $\mathcal{I}_{ij}$ and $\mathcal{I}_{i,j}$ are used interchangeably.
For convenience, we let $\mathcal{I}_{00} \coloneqq \emptyset$ and $\mathcal{I}_{i,n+1} \coloneqq \mathcal{I}_{i-1,i-1}$ for $i= 1,\dots,n$.
Under this notation, for $i = 1,\dots,n$ and $j = i,\dots,n$, we have
\begin{equation}
\mathcal{I}_{ij} = \mathcal{I}_{i,j+1} \cup \{(i,j)\}. \label{eq:I}
\end{equation}

\begin{lemma}\label{lem:face_CP}
Let $\mathcal{K} \subseteq \calS^n$ be a closed convex cone with $\mathcal{K} \subseteq \mathcal{N}^n$.
For $i = 1,\dots,n$ and $j = i,\dots,n$, the set $\mathcal{K}[\mathcal{I}_{ij}]$ is a face of $\mathcal{K}$.
\end{lemma}

\begin{proof}
It can be observed that $\mathcal{K}[\mathcal{I}_{ij}]$ is a convex subcone in $\mathcal{K}$.
Let $\bm{A},\bm{B}\in\mathcal{K}$.
Note that all the elements of $\bm{A}$ and $\bm{B}$ are nonnegative because $\mathcal{K} \subseteq \mathcal{N}^n$.
Now, we assume that $\bm{A} + \bm{B} \in \mathcal{K}[\mathcal{I}_{ij}]$.
Then for each $(k,l)\in \mathcal{I}_{ij}$, we have $A_{kl} + B_{kl} = 0$.
By the nonnegativity of $\bm{A}$ and $\bm{B}$, $A_{kl}$ and $B_{kl}$ must be zero.
Therefore, we obtain $\bm{A},\bm{B}\in\mathcal{K}[\mathcal{I}_{ij}]$.
\end{proof}

\begin{proposition}\label{prop:longest_chain_CP}
For any closed convex cone $\mathcal{K} \subseteq \calS^n$ with $\DD_+^n \subseteq \mathcal{K} \subseteq \mathcal{N}^n$, the following inclusions hold:
\begin{multline}
\underbrace{\mathcal{K}[\mathcal{I}_{nn}]}_{\text{$1$ face}} \subsetneq \underbrace{\mathcal{K}[\mathcal{I}_{n-1,n-1}] \subsetneq \mathcal{K}[\mathcal{I}_{n-1,n}]}_{\text{$2$ faces}} \subsetneq \mathcal{K}[\mathcal{I}_{n-2,n-2}] \subsetneq \cdots\\
\subsetneq\mathcal{K}[\mathcal{I}_{2n}] \subsetneq \underbrace{\mathcal{K}[\mathcal{I}_{11}] \subsetneq \cdots \subsetneq \mathcal{K}[\mathcal{I}_{1,n-1}] \subsetneq \mathcal{K}[\mathcal{I}_{1n}]}_{\text{$n$ faces}} \subsetneq \mathcal{K}[\mathcal{I}_{00}] = \mathcal{K}. \label{eq:chain_face_CP}
\end{multline}
In particular, we have $\ell_{\mathcal{K}} = T_n + 1$.
\end{proposition}

\begin{figure}[t]
\begin{center}
\includegraphics{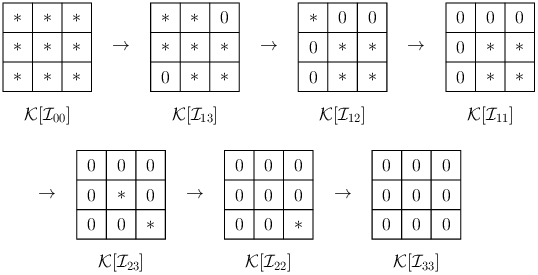}
\end{center}
\caption{The chain \eqref{eq:chain_face_CP} of faces of $\mathcal{K}$.
For each matrix belonging to each face $\mathcal{K}[\mathcal{I}_{ij}]$, the elements corresponding to the symbol $0$ must be zero, and those corresponding to the symbol $*$ can be nonzero.
}\label{fig:CP_face}
\end{figure}

Fig.~\ref{fig:CP_face} illustrates the chain~\eqref{eq:chain_face_CP} with $n = 3$.

\begin{figure}[t]
\begin{center}
\includegraphics{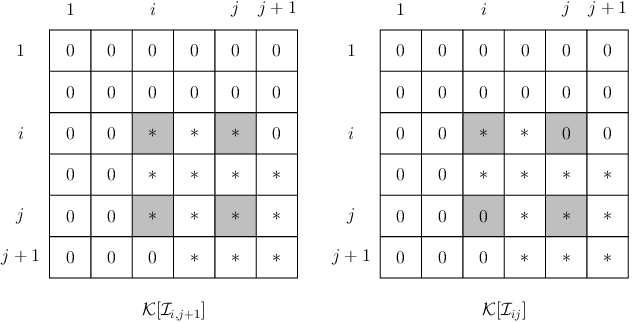}
\end{center}
\caption{Illustration of $\bm{E}_{ij} \in \mathcal{K}[\mathcal{I}_{i,j+1}] \setminus \mathcal{K}[\mathcal{I}_{ij}]$ in the proof of Proposition~\ref{prop:longest_chain_CP}.
For each matrix belonging to a face $\mathcal{K}[\mathcal{I}_{i,j+1}]$ or $\mathcal{K}[\mathcal{I}_{ij}]$, the elements corresponding to the symbol $0$ must be zero, and those corresponding to the symbol $*$ can be nonzero.
In addition, the grids where nonzero elements of the matrix $\bm{E}_{ij}$ exist are presented in gray.
On the one hand, because all the ``$0$'' grids in $\mathcal{K}[\mathcal{I}_{i,j+1}]$ are not gray, we have $\bm{E}_{ij} \in \mathcal{K}[\mathcal{I}_{i,j+1}]$.
On the other hand, as some of the ``$0$'' grids in $\mathcal{K}[\mathcal{I}_{ij}]$ are gray, we have $\bm{E}_{ij} \not\in \mathcal{K}[\mathcal{I}_{ij}]$.
}\label{fig:CP_proof_chain}
\end{figure}

\begin{proof}
For each $i = 1,\dots,n$ and $j = i,\dots,n$, \eqref{eq:I} implies that $\mathcal{K}[\mathcal{I}_{ij}] \subseteq \mathcal{K}[\mathcal{I}_{i,j+1}]$.
In what follows, we prove $\bm{E}_{ij} \in \mathcal{K}[\mathcal{I}_{i,j+1}] \setminus \mathcal{K}[\mathcal{I}_{ij}]$ to show the strictness of this inclusion.
Refer to Fig.~\ref{fig:CP_proof_chain} for a better understanding of the following proof.
We have $\bm{E}_{ij} \in \DD_+^n \subseteq \mathcal{K}$ and only the $(i,i)$th, $(j,j)$th, $(i,j)$th, and $(j,i)$th elements of $\bm{E}_{ij}$ are nonzero.
Therefore, for $\mathcal{I} \in \{\mathcal{I}_{i,j+1},\mathcal{I}_{ij}\}$, the matrix $\bm{E}_{ij}$ belongs to $\mathcal{K}[\mathcal{I}]$ if and only if all of the tuples $(i,i)$, $(j,j)$, and $(i,j)$ do not belong to $\mathcal{I}$.
On the one hand, by definition, $\bm{E}_{ij} \in \mathcal{K}[\mathcal{I}_{i,j+1}]$.
On the other hand, as $(i,j)\in \mathcal{I}_{ij}$, we have $\bm{E}_{ij} \not\in \mathcal{K}[\mathcal{I}_{ij}]$.

By $\mathcal{K} \subseteq \mathcal{N}^n$ and Lemma~\ref{lem:face_CP}, each set in the sequence~\eqref{eq:chain_face_CP} is a face of $\mathcal{K}$.
Because this chain comprises $(T_n + 1)$ faces of $\mathcal{K}$, we have $\ell_{\mathcal{K}} \ge T_n + 1$.
Combining this with Lemma~\ref{lem:length_upper_bound} yields $\ell_{\mathcal{K}} = T_n + 1$.
\end{proof}

\begin{corollary}\label{cor:dist_poly_CP}
For any closed convex cone $\mathcal{K} \subseteq \calS^n$ with $\CP^n \subseteq \mathcal{K} \subseteq \DNN^n$, we have $\ell_{\mathcal{K}} = T_n + 1$ and
\begin{equation}
\ell_{\rm poly}(\mathcal{K}) = \begin{cases}
0 & (n = 1),\\
T_n -2 & (n \ge 2).
\end{cases}\label{eq:dist_poly_CP}
\end{equation}
\end{corollary}

\begin{proof}
It follows from \eqref{eq:inclusion_CP} and Proposition~\ref{prop:longest_chain_CP} that $\ell_{\mathcal{K}} = T_n + 1$ with the longest chain~\eqref{eq:chain_face_CP}.
In what follows, we prove \eqref{eq:dist_poly_CP}.
Because the case of $n = 1$ follows from \eqref{enum:n=1} of Lemma~\ref{lem:dist_poly_upper_bound}, we only consider the case where $n \ge 2$.
The inclusion $\CP^n \subseteq \mathcal{K} \subseteq \DNN^n$ implies that
\begin{equation}
\CP^n[\mathcal{I}_{n-2,n-2}] \subseteq \mathcal{K}[\mathcal{I}_{n-2,n-2}] \subseteq \DNN^n[\mathcal{I}_{n-2,n-2}]. \label{eq:CP_K_DNN_face}
\end{equation}
From Lemma~\ref{lem:equation}, both of the first and third sets of \eqref{eq:CP_K_DNN_face} are equal to
\begin{equation}
\left\{\begin{pmatrix}
\bm{O} & \bm{O}\\
\bm{O} & \bm{A}
\end{pmatrix} \in\calS^n \relmiddle| \bm{A}\in\CP^2 \right\}. \label{eq:K_n-2_n-2}
\end{equation}
Therefore, the face $\mathcal{K}[\mathcal{I}_{n-2,n-2}]$ agrees with \eqref{eq:K_n-2_n-2}, which is not polyhedral since $\CP^2$ has infinitely many extreme rays~\cite[Remark~3.26]{SB2021}.
Given the fact that each face of a polyhedral cone is also polyhedral~\cite[page~172]{Rockafellar1970}, all the $(T_n - 2)$ faces $\mathcal{K}[\mathcal{I}_{n-2,n-2}] \subsetneq \cdots \subsetneq \mathcal{K}$ in the chain \eqref{eq:chain_face_CP} are not polyhedral.
Thus, we have $\ell_{\rm poly}(\mathcal{K}) \ge T_n -2$.
Combining this with \eqref{enum:n_ge_2} of Lemma~\ref{lem:dist_poly_upper_bound} yields the desired result.
\end{proof}

\begin{remark}\label{rem:difference}
Louren\c{c}o, Muramatsu, and Tsuchiya~\cite{LMT2018} have provided a chain of faces of $\DNN^n$ with length $T_n + 1$.
The chain is constructed by restricting first the $\frac{n(n-1)}{2}$ nondiagonal elements and then the $n$ diagonal elements to zero.
However, because of the existence of a polyhedral face of dimension $n$, the chain does not lead to $\ell_{\rm poly}(\DNN^n) = T_n -2$ when $n \ge 3$.
\end{remark}

\section{The copositive side}\label{sec:COP}
For $i = 1,\dots,n$ and $j = i,\dots,n$, let
\begin{equation*}
\mathcal{J}_{i,j} \coloneqq \bigcup_{k=1}^{i-1}\bigcup_{l=k}^n\{(k,l)\} \cup \bigcup_{l=i}^j\{(i,l)\}.
\end{equation*}
Here, $\mathcal{J}_{ij}$ and $\mathcal{J}_{i,j}$ are used interchangeably.
For convenience, we let $\mathcal{J}_{0n} \coloneqq \emptyset$ and $\mathcal{J}_{i,i-1} \coloneqq \mathcal{J}_{i-1,n}$ for $i = 1,\dots,n$.
Under this notation, for $i = 1,\dots,n$ and $j = i,\dots,n$, we have
\begin{equation}
\mathcal{J}_{ij} = \mathcal{J}_{i,j-1} \cup \{(i,j)\}. \label{eq:J}
\end{equation}

\begin{lemma}\label{lem:face_COP}
Let $\mathcal{K} \subseteq \calS^n$ be a closed convex cone with $\mathcal{K} \subseteq (\SDD_+^n)^*$.
For $i = 1,\dots,n$ and $j = i,\dots,n$, the set $\mathcal{K}[\mathcal{J}_{ij}]$ is a face of $\mathcal{K}$.
\end{lemma}
\begin{proof}
It can be observed that $\mathcal{K}[\mathcal{J}_{ij}]$ is a convex subcone in $\mathcal{K}$.
Let $\bm{A},\bm{B}\in \mathcal{K}$.
Note that $\bm{A}$ and $\bm{B}$ also belong to $(\SDD_+^n)^*$.
Now, we assume that $\bm{A} + \bm{B} \in \mathcal{K}[\mathcal{J}_{ij}]$.
Then for each $k = 1,\dots,i$, as $(k,k)\in\mathcal{J}_{ij}$, we observe that $A_{kk} + B_{kk} = 0$.
Because $A_{kk}$ and $B_{kk}$ are nonnegative (see \eqref{enum:dual_SDD+_diag} of Lemma~\ref{lem:dual_SDD+}), they must be zero.
This implies that for each $(k,l)\in \mathcal{J}_{ij}$, $A_{kl}$ and $B_{kl}$ are nonnegative (see \eqref{enum:dual_SDD+_nondiag} of Lemma~\ref{lem:dual_SDD+}).
Combining this with $A_{kl} + B_{kl} = 0$, we observe that $A_{kl}$ and $B_{kl}$ are also zero.
Thus, we obtain $\bm{A},\bm{B}\in\mathcal{K}[\mathcal{J}_{ij}]$.
\end{proof}

\begin{remark}
Lemma~\ref{lem:face_COP} does not hold for $\mathcal{K} = (\DD_+^n)^*$.
For example, let
\begin{equation*}
\bm{A} \coloneqq \begin{pmatrix}
0 & 1\\
1 & 2
\end{pmatrix}
,\ \bm{B} \coloneqq \begin{pmatrix}
0 & -1\\
-1 & 2
\end{pmatrix}.
\end{equation*}
Then it can be observed from \eqref{eq:dual_DD+} that $\bm{A},\bm{B} \in (\DD_+^2)^*$.
On the one hand, as
\begin{equation*}
\bm{A} + \bm{B} = \begin{pmatrix}
0 & 0\\
0 & 4
\end{pmatrix} \in (\DD_+^2)^*
\end{equation*}
and $A_{11} + B_{11} = A_{12} + B_{12} = 0$, it follows that $\bm{A} + \bm{B} \in (\DD_+^2)^*[\mathcal{J}_{12}]$.
On the other hand, $\bm{A},\bm{B} \not\in (\DD_+^2)^*[\mathcal{J}_{12}]$ since $A_{12},B_{12} \neq 0$.
Thus, $(\DD_+^2)^*[\mathcal{J}_{12}]$ is not a face of $(\DD_+^2)^*$.
\end{remark}

\begin{proposition}\label{prop:longest_chain_COP}
For any closed convex cone $\mathcal{K} \subseteq \calS^n$ with $\mathcal{N}^n \subseteq \mathcal{K} \subseteq (\SDD_+^n)^*$, the following inclusions hold:
\begin{multline}
\underbrace{\mathcal{K}[\mathcal{J}_{nn}]}_{\text{$1$ face}} \subsetneq \underbrace{\mathcal{K}[\mathcal{J}_{n-1,n}] \subsetneq \mathcal{K}[\mathcal{J}_{n-1,n-1}]}_{\text{$2$ faces}} \subsetneq \mathcal{K}[\mathcal{J}_{n-2,n}] \subsetneq \cdots\\
\subsetneq \mathcal{K}[\mathcal{J}_{22}] \subsetneq \underbrace{\mathcal{K}[\mathcal{J}_{1n}] \subsetneq \cdots \subsetneq \mathcal{K}[\mathcal{J}_{12}] \subsetneq \mathcal{K}[\mathcal{J}_{11}]}_{\text{$n$ faces}} \subsetneq \mathcal{K}[\mathcal{J}_{0n}] = \mathcal{K}. \label{eq:chain_face_COP}
\end{multline}
In particular, we have $\ell_{\mathcal{K}} = T_n + 1$.
\end{proposition}

\begin{figure}[t]
\begin{center}
\includegraphics{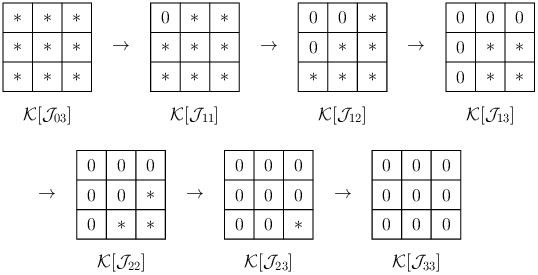}
\end{center}
\caption{The chain \eqref{eq:chain_face_COP} of faces of $\mathcal{K}$.
For each matrix belonging to each face $\mathcal{K}[\mathcal{J}_{ij}]$, the elements corresponding to the symbol $0$ must be zero, and those corresponding to the symbol $*$ can be nonzero.
}\label{fig:COP_face}
\end{figure}

Fig.~\ref{fig:COP_face} illustrates the chain~\eqref{eq:chain_face_COP} with $n = 3$.

\begin{figure}[t]
\begin{center}
\includegraphics{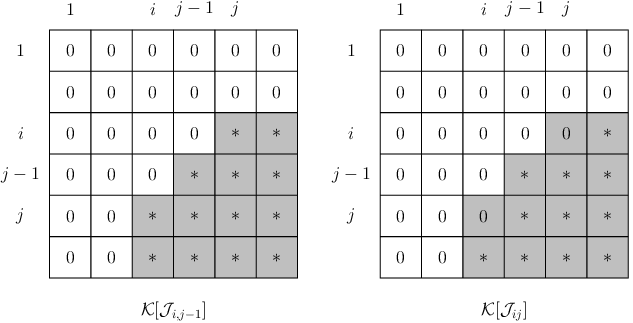}
\end{center}
\caption{Illustration of $\bm{E}[\mathcal{J}_{i,j-1}] \in \mathcal{K}[\mathcal{J}_{i,j-1}] \setminus \mathcal{K}[\mathcal{J}_{ij}]$ in the proof of Proposition~\ref{prop:longest_chain_COP}.
For each matrix belonging to a face $\mathcal{K}[\mathcal{J}_{i,j-1}]$ or $\mathcal{K}[\mathcal{J}_{ij}]$, the elements corresponding to the symbol $0$ must be zero, and those corresponding to the symbol $*$ can be nonzero.
In addition, the grids where nonzero elements of the matrix $\bm{E}[\mathcal{J}_{i,j-1}]$ exist are presented in gray.
On the one hand, as all the ``$0$'' grids in $\mathcal{K}[\mathcal{J}_{i,j-1}]$ are not in gray, we have $\bm{E}[\mathcal{J}_{i,j-1}]\in \mathcal{K}[\mathcal{J}_{i,j-1}]$.
On the other hand, as some of the ``$0$'' grids in $\mathcal{K}[\mathcal{J}_{ij}]$ are gray, we have $\bm{E}[\mathcal{J}_{i,j-1}] \not\in \mathcal{K}[\mathcal{J}_{ij}]$.
}\label{fig:COP_proof_chain}

\end{figure}

\begin{proof}
For each $i = 1,\dots,n$ and $j = i,\dots,n$, \eqref{eq:J} implies that $\mathcal{K}[\mathcal{J}_{ij}] \subseteq \mathcal{K}[\mathcal{J}_{i,j-1}]$.
In what follows, we prove $\bm{E}[\mathcal{J}_{i,j-1}] \in \mathcal{K}[\mathcal{J}_{i,j-1}] \setminus \mathcal{K}[\mathcal{J}_{ij}]$ to show the strictness of this inclusion.
Refer to Fig.~\ref{fig:COP_proof_chain} for a better understanding of the following proof.
Because $\bm{E}[\mathcal{J}_{i,j-1}] \in \mathcal{N}^n \subseteq \mathcal{K}$, it suffices to check whether or not the $(k,l)$th element of the matrix $\bm{E}[\mathcal{J}_{i,j-1}]$ is zero for each $(k,l)$ belonging to $\mathcal{J} \in \{\mathcal{J}_{i,j-1},\mathcal{J}_{ij}\}$.
On the one hand, as the $(k,l)$th element of $\bm{E}[\mathcal{J}_{i,j-1}]$ is zero for all $(k,l)\in \mathcal{J}_{i,j-1}$, we have $\bm{E}[\mathcal{J}_{i,j-1}] \in \mathcal{K}[\mathcal{J}_{i,j-1}]$.
On the other hand, as $(i,j)\in \mathcal{J}_{ij}$ and the $(i,j)$th element of $\bm{E}[\mathcal{J}_{i,j-1}]$ is one, we have $\bm{E}[\mathcal{J}_{i,j-1}] \not\in \mathcal{K}[\mathcal{J}_{ij}]$.

By $\mathcal{K} \subseteq (\SDD_+^n)^*$ and Lemma~\ref{lem:face_COP}, each set in the sequence~\eqref{eq:chain_face_COP} is a face of $\mathcal{K}$.
Because this chain is composed of $(T_n + 1)$ faces of $\mathcal{K}$, we have $\ell_{\mathcal{K}} = T_n + 1$.
\end{proof}

\begin{corollary}\label{cor:dist_poly_COP}
For any closed convex cone $\mathcal{K} \subseteq \calS^n$ with $\SPN^n \subseteq \mathcal{K} \subseteq \COP^n$, we have $\ell_{\mathcal{K}} = T_n + 1$ and
\begin{equation}
\ell_{\rm poly}(\mathcal{K}) = \begin{cases}
0 & (n = 1),\\
T_n -2 & (n \ge 2).
\end{cases} \label{eq:dist_poly_COP}
\end{equation}
\end{corollary}

\begin{proof}
It follows from \eqref{eq:inclusion_COP} and Proposition~\ref{prop:longest_chain_COP} that $\ell_{\mathcal{K}} = T_n + 1$ with the longest chain~\eqref{eq:chain_face_COP}.
In what follows, we prove \eqref{eq:dist_poly_COP}.
Since the case $n = 1$ follows from \eqref{enum:n=1} of Lemma~\ref{lem:dist_poly_upper_bound}, we only consider the case of $n \ge 2$.
The inclusion $\SPN^n \subseteq \mathcal{K} \subseteq \COP^n$ implies that
\begin{equation}
\SPN^n[\mathcal{J}_{n-2,n}] \subseteq \mathcal{K}[\mathcal{J}_{n-2,n}] \subseteq \COP^n[\mathcal{J}_{n-2,n}]. \label{eq:SPN_K_COP_face}
\end{equation}
From Lemma~\ref{lem:equation}, both the first and third sets of \eqref{eq:SPN_K_COP_face} are equal to
\begin{equation}
\left\{\begin{pmatrix}
\bm{O} & \bm{O}\\
\bm{O} & \bm{A}
\end{pmatrix} \in\calS^n \relmiddle| \bm{A}\in\COP^2 \right\}. \label{eq:K_n-2_n}
\end{equation}
Therefore, the face $\mathcal{K}[\mathcal{J}_{n-2,n}]$ agrees with \eqref{eq:K_n-2_n},
which is not polyhedral as $\COP^2$ has infinitely many extreme rays~\cite[Theorem~2.29]{SB2021}.
Therefore, we obtain $\ell_{\rm poly}(\mathcal{K}) = T_n -2$.
\end{proof}

\begin{remark}
The order in which the elements of a matrix are restricted to zero is essential for obtaining the longest chain of faces.
Let us consider what would happen if we used $\mathcal{I}_{ij}$, defined in Section~\ref{sec:CP}, instead of $\mathcal{J}_{ij}$ in the discussion of this section.
Then for a closed convex cone $\mathcal{K}$ with $\SPN^n \subseteq \mathcal{K} \subseteq \COP^n$, the set $\mathcal{K}[\mathcal{I}_{ij}]$ is not necessarily a face of $\mathcal{K}$.
For example, consider the set $\mathcal{K}[\mathcal{I}_{1n}]$.
Let
\begin{equation*}
\bm{A} \coloneqq \begin{pmatrix}
1 &  &  -1\\
&  \bm{O} &  \\
-1 &  &  1
\end{pmatrix},\ \bm{B} \coloneqq \begin{pmatrix}
1 &  &  1\\
&  \bm{O} &  \\
1 &  &  1
\end{pmatrix} \in \calS_+^n \subseteq \mathcal{K}.
\end{equation*}
On the one hand, since
\begin{equation*}
\bm{A} + \bm{B} = \begin{pmatrix}
2 &  &  0\\
&  \bm{O} &  \\
0 &  &  2
\end{pmatrix} \in \mathcal{K}
\end{equation*}
and $A_{1n} + B_{1n} = 0$, it follows that $\bm{A} + \bm{B} \in \mathcal{K}[\mathcal{I}_{1n}]$.
On the other hand, $\bm{A},\bm{B}\not\in \mathcal{K}[\mathcal{I}_{1n}]$ since $A_{1n},B_{1n} \neq 0$.
Thus, $\mathcal{K}[\mathcal{I}_{1n}]$ is not a face of $\mathcal{K}$.

Conversely, for a closed convex cone $\mathcal{K}$ with $\CP^n \subseteq \mathcal{K} \subseteq \DNN^n$ and the set $\mathcal{J}_{ij}$, $\mathcal{K}[\mathcal{J}_{ij}]$ is indeed a face of $\mathcal{K}$.
However, for each $i = 1,\dots,n$, all the inclusions in $\mathcal{K}[\mathcal{J}_{in}] \subseteq  \cdots \subseteq \mathcal{K}[\mathcal{J}_{ii}]$ hold with equality.
This is because for any $\bm{A} \in \DNN^n$, $A_{ii} = 0$ implies that $A_{ij} = 0$ for all $j = 1,\dots,n$.
Therefore, if we use $\mathcal{J}_{ij}$ instead of $\mathcal{I}_{ij}$ in the discussion of Section~\ref{sec:CP}, we cannot obtain a chain of faces of $\mathcal{K}$ with length $T_n + 1$.
\end{remark}

\section{Conclusion}\label{sec:conclusion}
In this paper, we demonstrated the construction of a chain of faces of length $\ell_{\mathcal{K}} = T_n + 1$ for any closed convex cone $\mathcal{K}$ satisfying either $\DD_+^n \subseteq \mathcal{K} \subseteq \mathcal{N}^n$ or $\mathcal{N}^n \subseteq \mathcal{K} \subseteq (\SDD_+^n)^*$, where $T_n$ is defined as $\frac{n(n+1)}{2}$.
Furthermore, for any closed convex cone $\mathcal{K}$ satisfying $\CP^n \subseteq \mathcal{K} \subseteq \DNN^n$ or $\SPN^n \subseteq \mathcal{K} \subseteq \COP^n$ with $n\ge 2$, we established that $\ell_{\rm poly}(\mathcal{K}) = T_n - 2$ as well as $\ell_{\mathcal{K}} = T_n + 1$, which are the maximum possible for any closed convex cone in $\calS^n$.
Notably, such cones $\mathcal{K}$ include not only the completely positive, doubly nonnegative, SPN, and copositive cones but also the closure of the completely positive semidefinite cone and some approximation hierarchies.

As mentioned in Section~\ref{sec:intro}, the value $\ell_{\rm poly}(\mathcal{K}) + 1$ is an upper bound for the singularity degree of a linear conic feasibility problem over $\mathcal{K}$.
However, this bound is not tight in general.
For example, the singularity degree of a linear conic feasibility problem over $\DNN^n$ is bounded by $n$~\cite[Corollary~20]{LMT2018}, which is smaller than $\ell_{\rm poly}(\DNN^n) + 1 = T_n - 1$ when $n \ge 3$.
It is an open problem to determine whether the upper bound for the singularity degree is tight or not for other cones.

\vspace{0.5cm}
\noindent
{\bf Acknowledgments}
The author would like to thank Bruno F. Louren\c{c}o for useful discussions.
This work was supported by Japan Society for the Promotion of Science Grants-in-Aid for Scientific Research (Grant Number JP22KJ1327).

\section*{Declarations}
{\bf Conflict of interest}
The author declares that there are no competing interests.

\end{document}